\documentclass[11pt]{amsart}

\usepackage{mathrsfs}
\newtheorem{theorem}{Theorem}
\newtheorem{lemma}{Lemma}
\newtheorem{corollary}{Corollary}
\newtheorem{proposition}{Proposition}
\newtheorem{definition}{Definition}

\newtheorem{example}{Example}
\begin{document}
\title[Isoparametric hypersurfaces]{Ideal theory and classification of isoparametric hypersurfaces}
\author{Quo-Shin Chi}

\address{Department of Mathematics, Washington University, St. Louis, MO 63130}
\email{chi@math.wustl.edu}
\begin{abstract} The classification of isoparametric hypersurfaces with four principal curvatures in the sphere interplays
in a deep fashion with commutative algebra, whose
abstract and comprehensive nature might obscure a differential
geometer's insight into the classification problem that encompasses a wide
spectrum of geometry and topology.
In this paper, we make an effort to bridge the gap by walking through the
important part of commutative algebra
central to the classification of such hypersurfaces, such that all the
essential ideal-theoretic ingredients are laid out in a way as much intuitive,
motivating and geometric with rigor maintained as possible. We then explain how we developed the
technical side of the entailed ideal theory,
pertinent to isoparametric hypersurfaces with four principal curvatures, for the classification done in our papers~\cite{CCJ},~\cite{Ch1}
and~\cite{Ch3}.
\end{abstract}
\maketitle
\pagestyle{myheadings}
\markboth{QUO-SHIN CHI}{ISOPARAMETRIC HYPERSURFACES}

\section{Introduction} 
An isoparametric hypersurface $M$ in the sphere is one
whose principal curvatures and their multiplicities are fixed constants.
The classification of such hypersurfaces has been an outstanding problem in
submanifold geometry, listed as Problem 34 in~\cite{Ya}, as can be witnessed by its long history. 
See Section~\ref{ISO} for more background details.

The story started with Cartan's seminal investigation and complete classification when $g$, the number of principal curvatures,
is $\leq 3$~\cite{Car1},~\cite{Car2},~\cite{Car3},~\cite{Car4}, 
followed by M\"{un}zner's remarkable structure theory~\cite{Mu} to lay the 
groundwork for the two classes of first known inhomogeneous examples with $g=4$
constructed by Ozeki and Takeuchi~\cite[I]{OT}, which was then generalized to infinite classes of inhomogeneous examples by Ferus, Karcher and
M\"{u}nzner~\cite{FKM}. 

Among other things, M\"{u}nzner~\cite[II]{Mu} established
that $g=1,2,3,4$ or 6. Thanks to the subsequent work of Abresch~\cite{A},
who identified the only two possible multiplicity pairs of the principal curvatures when $g=6$, Dorfmeister and Neher~\cite{DN1} succeeded in the classification
in the case of the smaller pair ($=(1,1)$), and recently Miyaoka~\cite{Mi},~\cite{Mi1} settled the case of the other pair ($=(2,2)$); the isoparametric 
hypersurfaces are homogeneous.

It is worth pointing out
that isoparametric submanifolds in the sphere were introduced by
Terng~\cite{Te} and later those of codimension $\geq 2$ were all classified to be homogeneous by
Thorbergsson~\cite{Th}. Thorbergsson's method was to associate the submanifold with a Tits
building to employ the rigidity of Tits buildings of rank $\geq 3$ in the classification. Though
Immervoll~\cite{Im} proved that an isoparametric hypersurface with four principal curvatures
in the sphere also gives rise to an incidence structure which is a
Tits building, it cannot be applied directly
to the classification as in Thorbergsson's approach since there is
no such classification of rank 2.

As of this writing, for $g=4$, there remains the last unsettled case with
multiplicity pair $(7,8)$. The classification enjoys a deep interaction with
a major part of the ideal theory in commutative algebra, whose abstract
and comprehensive nature
might obscure a differential geometer
when facing a classification problem of the sort such as isoparametric hypersurfaces, 
that encompasses a wide
spectrum of geometry and topology.

The purpose of the paper is twofold. On the one hand, we will walk through the
important part of commutative algebra central to the classification of
isoparametric hypersurfaces with four principal curvatures, in as much intuitive,
motivating and geometric a way with the rigor of the
presentation maintained as possible. On the other hand, with a good 
look at the entailed ideal theory we will then explain 
its technical side we developed
in~\cite{CCJ},~\cite{Ch1} and~\cite{Ch3} on which the classification hinges.

We hope the paper can bring the reader to a further appreciation of the breadth
and depth of the intriguing classification story of isoparametric hypersurfaces.

\section{A walk through some ideal theory} 

\subsection{Codimension 1 estimate and reducedness}
Let ${\mathbb C}^n$ be parametrized by $z_1,\cdots,z_n$, and let $V$ be a
{\em variety} in ${\mathbb C}^n$, i.e., a set defined by the common zeros of
$m+1$ polynomials
$p_0,p_1,\cdots,p_m$ in the polynomial ring $P[n]$ in the variables
$z_1,\cdots,z_n$. Hilbert's basis theorem~\cite[p. 13]{Fu} implies that
all ideals of $P[n]$ are finitely generated. Moreover, Hilbert's Nullstellensatz~\cite[p. 20]{Fu} states that $f\in P[n]$ vanishes on $V$
if and only if $f^n$, for some positive integer $n$, belongs to the ideal $I\subset P[n]$ generated by
$p_0,\cdots,p_m$, denoted by $(p_0,\cdots,p_m)$ henceforth.
In particular, if we let ${\mathcal O}(V)$ be the ideal of $P[n]$ of all
polynomial functions vanishing on $V$, also called the {\em coordinate ring} of $V$,
then there is a one-to-one correspondence between a variety $V$
and its coordinate ring ${\mathcal O}(V)$ in ${\mathbb C}^n$.

In general, $V$ may have finitely many irreducible components
$V_1,\cdots,V_s$
which cannot be further decomposed into unions of
varieties,
a consequence of Hilbert's basis theorem~\cite[pp. 15-16]{Fu}. Dually,
${\mathcal O}(V)$ is the intersection of finitely many prime ideals

\begin{equation}\label{primary}
{\mathcal O}(V)=\cap_{j=1}^s {\mathcal P}_j,
\end{equation}
where ${\mathcal P}_j$ is the ideal of $f\in P[n]$ vanishing on $V_j$. (Recall
an ideal ${\mathcal P}$ is prime if $ab\in{\mathcal P}$ implies either $a$ or $b$
is in ${\mathcal P}$.) Each ${\mathcal P}_j$ is a minimal prime ideal containing ${\mathcal O}(V)$
since $V_j$ is contained in no other irreducible varieties contained in $V$. 
On the other hand, each $V_j$ is a complex manifold
away from its {\em singular set}, which is itself a variety of a smaller
dimension where $V_j$ is not manifold-like. In addition,
there is another type of singular points of $V$, namely, those
which lie in the intersection of two irreducible components where $V$ is not
manifold-like. Together, the
two types of
points constitute the singular set ${\mathcal S}(V)$ of $V$. Explicitly,
\begin{equation}\label{Dim}
{\mathcal S}(V)=(\cup_{i\neq j}(V_i\cap V_j))\cup(\cup_{j}{\mathcal S}(V_j)),
\end{equation}
where if the coordinate ring ${\mathcal P}_j$ of $V_j$ is generated by the
polynomials $q_1,\cdots,q_l$, we let
$$
\text{edim}(z):=n-\text{rank}(\partial(q_1,\cdots,q_l)/
\partial(z_1,\cdots,z_n)),
$$
be the {\em embedding dimension} that is the natural dimension one expects from
the implicit function theorem in calculus. Then
$$
\dim(V_j)=\inf_{z\in V_j}\text{edim}(z),\quad{\mathcal S}(V_j)=\{z\in V_j:\text{edim}(z)>\dim(V_j)\}.
$$
See~\cite[p. 170]{Ku} for~\eqref{Dim} that is even true on the ideal level.

\begin{example} Consider the polynomial
$$
p=(x-1)(y^2-x^2(x+1))
$$
over ${\mathbb C}^2$. The variety $p=0$ consists of two irreducible
components $V_1$ and $V_2$, which are respectively the zeros sets $x-1=0$
and $y^2-x^2(x+1)=0$. The singular set of $V$ consists of the singular point
of $V_2$, which is $(0,0)$, and $(1,\pm\sqrt{2})$, the two points of
intersection of $V_1$ and $V_2$.
\end{example}

Here comes the subtlety. In general ${\mathcal O}(V)$ properly contains $I$
that defines the variety $V$.

\begin{example}\label{double} Consider $p_0(x,y)=y-x^2$ and $p_1(x,y)=y$. Their common
zero set
$V$ is $\{(0.0)\}$. The polynomial $x$ vanishes on $V$, i.e.,
$x\in {\mathcal O}(V)$. However, $x$
does not belong to the ideal $I=(p_0,p_1)$, as can be easily verified.
Instead, $x^2$ lies in $I$.
\end{example}

For an ideal $I$, we denote
by $\sqrt{I}$ the {\em radical} of $I$ consisting of $f\in P[n]$ such that
$f^n\in I$ for some positive integer $n$. A fundamental question is:

\vspace{2mm}

\noindent Under what condition $I$, which defines $V$, is exactly ${\mathcal O}(V)$?

\vspace{2mm}

Clearly, a necessary and sufficient condition is that $f$ vanishes on
$V$ implies $f$ lies in $I$. Alternatively put, by Nullstellensatz,
$f^n\in I$
for some $n$ implies $f\in I$, i.e., $\sqrt{I}=I$, in which case $I$ is
called a {\em radical ideal} and $P[n]/I$ interchangeably is called
a {\em reduced ring}, for reason that it thus has no nilpotent elements, i.e.,
no $r\neq 0$ for which $r^n=0$ for some $n$.

Note that $I$ is radical when $I$ is a prime ideal,
or equivalently,
when the variety $V$ defined by $I$ is irreducible. The second fundamental question
is:

\vspace{2mm}

\noindent Under what condition is $I$ a prime ideal?

\vspace{2mm}

To answer the first question, let us observe that if $I$ is radical, i.e., if
$I={\mathcal O}(V),$ then by~\eqref{primary}, we must have

\vspace{2mm}

\noindent $(\dagger)$ $I$ is the intersection of only those minimal prime ideals containing I.

\vspace{2mm}

Not all ideals are the intersection of only minimal prime ideals containing $I$.

\begin{example} Consider $I=(x^2,xy)$ in $P[2]$. It is easily seen that
$$
I=(x)\cap (x^2,y).
$$
Since the variety defined by $I$ is the $y$-axis, the only minimal ideal containing
$I$ is $(x)$. Note that $(x)=\sqrt{I}$.
\end{example}

In addition, we must also have
                                               
\vspace{2mm}

\noindent $(\ddagger)$ For each $m$ in a minimal prime ideal ${\mathcal P}$
containing $I$,
there is an $s$ of $P[n]$ not in ${\mathcal P}$
such that $sm\in I$.

\vspace{2mm}

In fact, for each $m\in{\mathcal P}_1$,
pick an $s\in(\cap_{j=2}^s{\mathcal P}_j)\setminus{\mathcal P}_1$. Then $s$ is a polynomial
vanishing on $\cup_{j=2}^s V_j$ but not on $V_1$. We have $sm=0$ on $V$, etc.

It turns out that $(\dagger)$ and $(\ddagger)$ are also sufficient to imply
that $I$ is radical, called Serre's $(S_1,R_0)$ criterion.
One needs to establish that $\sqrt{I}=I$. To this end, on the one hand 
$(\dagger)$ and $(\ddagger)$ ensure that $I$ is the intersection of minimal prime ideals containing $I$~\cite[p. 181]{Ku}. On the other hand, 
it is well known~\cite[p. 71]{Ei} that $\sqrt{I}$ is the intersection of all prime ideals, and so in particular, is the intersection of all minimal prime ideals.
Thus $\sqrt{I}=I$. 

So now it comes down to asking when $(\dagger)$ and $(\ddagger)$ hold true.
A broad category in which $(\dagger)$ is valid is when the generators $p_0,\cdots,p_m$ of $I$ form a
{\em regular sequence}, a notion central in commutative algebra that generalizes that
of smooth transversal intersections.

Recall that in a ring an element $a\neq 0$ is called a
{\em zero divisor} if $ab=0$ for some element $b\neq 0$. Otherwise,
it is called a {\em non-zerodivisor}.

\begin{definition}\label{def} 
A regular sequence in the polynomial ring $P[n]$
is a sequence $p_0,\cdots,p_k$ in $P[n]$ such that firstly the variety
defined by $p_0=\cdots=p_k=0$ in ${\mathbb C}^n$ is not empty. Moreover,
$p_i$ is 
a non-zerodivisor in the quotient ring $P[n]/(p_0,\cdots,p_{i-1})$ for
$1\leq i\leq k$;
in other words, any relation
$$
p_1f_1+\cdots+p_{i-1}f_{i-1}+p_if_i=0
$$
will result in $f_i$ being in the form
$$
f_i=p_0h_0^i+\cdots+p_{i-1}h_{i-1}^i
$$
for some $h_0^i,\cdots,h_{i-1}^i\in P[n]$ for $1\leq i\leq k$.
\end{definition}
Thus a regular sequence imposes strong algebraic independence amongst its
elements. We shall return to this later.

\begin{example}\label{reg}
A single nonconstant $p\in P[n]$ forms a regular sequence, because by Nullstellensatz
$p=0$ is nonempty, which is the only non-void condition in the definition of a regular
sequence.

Two homogeneous and relatively prime polynomials $p$ and $q$ of degree $\geq 1$
form a regular sequence. Firstly, $p=q=0$ is nonempty since 0 is clearly a solution.
Secondly, $pf_1+qf_2=0$ implies $f_2=ph$ since $p$ and $q$ are relatively prime.

More generally, any two relatively prime polynomials $p$ and $q$ with a nonempty
common zero set form a regular sequence.
\end{example}

\begin{example}\label{coordinate} The first $k$ coordinates $z_1,\cdots,z_k$ of\, ${\mathbb C}^n$
form a regular sequence for any $k$. To see this, first of all $z_1=\cdots=z_k=0$
is not empty. Next, if
$$
z_1f_1+z_2f_2=0,\quad\text{or}\;\;z_2f_2=-z_1f_1,
$$     
then since $z_2$ does not vanish identically on the hyperplane $z_1=0$, it must be that
$f_2$ does, so that $f_2=z_1g_1$. Similarly, if
$$
z_3f_3=-z_1f_1-z_2f_2,
$$
then $f_3$ must vanish identically on the linear subspace $z_1=z_2=0$,
which ensures that
$f_3=z_1h_1+z_2h_2,$ etc.
\end{example}

\begin{example} $p=xz$ and $q=yz$ in\, ${\mathbb C}^3$ do not form a regular sequence.
This is because $py-qx=0$ and $x$ is not a multiple of\, $p$. Note that $p=q=0$
is the variety\, $V$ consisting of the plane $z=0$ and the line $x=y=0$.
$V$ is not of pure dimension.
\end{example}

It is a deep fact that the variety $V$ defined by a regular sequence $p_0,\cdots,p_m$ in
${\mathbb C}^n$ is of pure dimension $n-m-1$. It is not just that the manifold part
of each irreducible component of $V$ is of the right dimension $n-m-1$. What is remarkable
is that it is the right dimension at each singular point as well,
more generally so on the ideal level!
The technical and deep concept
entailed here is {\em Cohen-Macaulayness}. That $(\dagger)$ holds for
an ideal
$I\subset P[n]$ generated by a regular sequence is a consequence of this property
of pure dimension, on the ideal level, in the context of
Macaulay Unmixedness Theorem~\cite[p. 187]{Ku}.

Having set aside $(\dagger)$, let us turn to $(\ddagger)$. We now express it
in terms of the ring $R:=P[n]/I$ itself to make the statement intrinsic.

\vspace{2mm}

\noindent $(\ddagger')$ For each $m$ in a minimal prime ideal ${\mathcal P}$
in $R$,
there is an $s\in R\setminus{\mathcal P}$
such that $sm=0$.

\vspace{2mm}

Before proceeding further, let us look at
Example~\ref{double} once more. In the example,
since the ideal $I$ is generated by
$y-x^2$ and $y$, or equivalently by $x^2$ and $y$, the quotient ring is thus
$$
R=P[2]/I=\{a+bx:a,b\in{\mathbb C}, x^2=0\}.
$$
The minimal prime ideal ${\mathcal P}:=(x)$ in $R$ fails to satisfy 
$(\ddagger')$. Indeed, the only $s\in R\setminus {\mathcal P}$ is a nonzero constant in ${\mathbb C}$,
whose product with $x$ can never be zero. Note that this example satisfies $(\dagger)$
as $p_0$ and $p_1$ in the example form a regular sequence.

Now, it is a pleasant fact that the implicit function theorem
comes to the rescue to resolve $(\ddagger')$. This is known as
Serre's criterion of reducedness~\cite[p. 462]{Ei}.

\begin{theorem}\label{Serre} {\rm (Serre)} Let $I$ be the ideal generated by
a regular sequence $p_0,\cdots,p_m,m+1\leq n,$ in $P[n]$
that define the variety $V$. Let $J$ be the subvariety of\, $V$
consisting of all points of\, $V$ where the Jacobian matrix
$$
\partial (p_0,\cdots,p_m)/\partial(z_1,\cdots,z_n)
$$
is not of full rank $m+1$. Suppose the codimension of $J$ is $\geq 1$ in\, $V$.
Then $R:=P[n]/I$ is reduced.
\end{theorem}

Before we outline the idea of the proof of Serre's criterion of reducedness,
let us first remark that $(\ddagger')$
can be further transformed into a statement in terms of the important
concept of {\em localization}
in commutative algebra.

\begin{definition} Let $R$ be a commutative ring with identity, and let $S\setminus\{0\}$
be a {\em multiplicatively closed} subset of $R$ in the sense that $ab\in S$ for $a$ and $b$
in $S$. We define $R_S$ to be the ring
$$
R_S:=\{r/s: s\in S\}.
$$
Here, $r/s$ is the equivalence class of pairs $(r,s)$ subject to the relation
$(r_1,s_1)\sim(r_2,s_2)$ if there is an $t\in S$ such that
$t(r_1s_2-r_2s_1)=0$.
\end{definition}
The extra $t$ in the definition
is to ensure $r_1/s_1=tr_2/ts_2$ if $r_1/s_1=r_2/s_2$.

\begin{example}\label{primeideal} When $S=R\setminus{\mathcal P}$ for a prime ideal
${\mathcal P}$,
the ring $R_S$ is denoted instead by $R_{\mathcal P}$.

$R_{\mathcal P}$ is a {\em local ring} in the sense that ${\mathcal P}_{\mathcal P}$
is its unique maximal ideal. To see this, we observe that
$r\in R_{\mathcal P}\setminus {\mathcal P}_{\mathcal P}$ if and only if\, $r$
is a {\em unit} in
$R_{\mathcal P}$ (a unit $a$ is one such that $ab=1$ for some $b$). (Reason: $r=a/b$ with $a,b\in R\setminus {\mathcal P}$ so that
$(a/b) (b/a) =1$, and vice versa.) Moreover,
any proper ideal $I$ in $R_{\mathcal P}$ can never admit any unit, and so
$I$ must be contained in the ideal ${\mathcal P}_{\mathcal P}$.

$r/s$ is regarded as a "rational function" of $r$ divided by $s$,
where $s$ does not vanish
on the irreducible variety defined by ${\mathcal P}$.
\end{example}

\begin{example} Recall that a commutative ring $R$ with identity is
a {\em domain} if it has no zero divisors.
For an ideal $I$ of $R$, the ring $R/I$ is a domain
if and only if $I$ is a prime ideal.

Assume $R$ is a domain. Let $S:=R\setminus\{0\}$. Then $S$ is multiplicatively
closed. $R_S$ is a field called the {\em quotient field} of $R$.

Note that $R_{\mathcal P}/{\mathcal P}_{\mathcal P}$
is exactly the quotient field $\kappa({\mathcal P})$ of the domain
$R/{\mathcal P}$ via the map
$$
r/s\in R_{\mathcal P}\longmapsto (r+{\mathcal P})/(s+{\mathcal P})\in\kappa({\mathcal P}).
$$
\end{example}

\begin{example} More generally, let $R$ be a commutative ring with identity, and let $S$ be its
subset of non-zerodivisors. Then $S$ is multiplicatively closed. $R_S$
is called the {\em quotient ring} of $R$, denoted by $Q(R)$.
\end{example}

With Example~\ref{primeideal}, $(\ddagger')$ can be rephrased as

\vspace{2mm}

\noindent $(\bullet)$ The maximal ideal ${\mathcal P}_{\mathcal P}=0$.

\vspace{2mm}

\begin{example} Let us look at Example~\eqref{coordinate}. The prime ideal ${\mathcal P}=(z_1,\cdots,z_k)$
define the linear subspace $z_1=\cdots=z_k=0$. Let $x=(z_{k+1},\cdots,z_n)$
and $y=(z_1,\cdots,z_k)$. Any polynomial $f$ can be Taylor expanded as
\begin{equation}\label{taylor}
f(x,y)=f_0(x)+f_1(x)y+f_2(x)y^2+\cdots
\end{equation}
with the obvious shorthand notation. Now, $P[n]_{\mathcal P}$ is the set of
all rational functions $f/g$ with $f$ and $g$ given as in~\eqref{taylor} and
$g_0\neq 0$,
while ${\mathcal P}_{\mathcal P}$ consists of $f/g$ in $P[n]_{\mathcal P}$
with $f_0=0$. $P[n]_{\mathcal P}/{\mathcal P}_{\mathcal P}$ is the quotient
field $\kappa({\mathcal P})$ consisting of rational
functions of the form $f_0(x)/g_0(x)$ with $g_0\neq 0$.
\end{example}

\begin{example}\label{jacobian}
Continuing with the preceding example, for $f/g\in {\mathcal P}_{\mathcal P}$
with $f_0=0$, let us take the first differential restricted to $y=0$ to obtain

\begin{equation}\label{toy}
{\mathcal P}_{\mathcal P}\longrightarrow \Omega^1(P[n]_{\mathcal P})|_{y=0},
\quad \frac{f}{g}\longmapsto  d(\frac{f}{g})|_{y=0} = \frac{f_1dy}{g_0},
\end{equation}
whose kernel consists of
$$
\frac{f}{g},\quad f=f_2y^2+f_3y^3+\cdots=y^2h\;\;\text{for some}\;\; h \;\;\text{so that}\;\;\frac{f}{g}\in({\mathcal P}_{\mathcal P})^2.
$$
Therefore, we have the injection
$$
0 \longrightarrow {\mathcal P}_{\mathcal P}/({\mathcal P}_{\mathcal P})^2\stackrel{D}
\longrightarrow\Omega^1(P[n]_{\mathcal P})|_{y=0},
$$
where $D$ is induced by $d$. On the other hand, We have the natural projection
\begin{eqnarray}\label{TOY}
\aligned
&\Omega^1(P[n]_{\mathcal P})|_{y=0}\stackrel{\pi}\longrightarrow
\Omega^1(\kappa({\mathcal P}))\longrightarrow 0,\\
&d(\frac{f}{g})|_{y=0}=d(\frac{f_0}{g_0})+\frac{g_0f_1-f_0g_1}{g_0^2}dy
\longmapsto d(\frac{f_0}{g_0}),
\endaligned
\end{eqnarray}
so that in fact we arrive at the exact sequence (called the {\em conormal} sequence)

\begin{equation}\label{exact}
0 \longrightarrow {\mathcal P}_{\mathcal P}/({\mathcal P}_{\mathcal P})^2\stackrel{D}
\longrightarrow\Omega^1(P[n]_{\mathcal P})|_{y=0}
\stackrel{\pi}\longrightarrow\Omega^1(\kappa({\mathcal P}))\longrightarrow 0
\end{equation}
considered as vector spaces over the field $\kappa({\mathcal P})$.
\end{example} 

More generally, for $R=P[n]/I$ with $I=(p_0,\cdots,p_m)$,
consider the first differential
$$
I\stackrel{d}\longrightarrow R\otimes_{P[n]}\Omega^1(P[n]),
$$
where
$$
d:p_i\longmapsto 1\otimes dp_i
=1\otimes \sum_j\frac{\partial p_i}{\partial z_j} dz_j =
\sum_j\frac{\partial p_i}{\partial z_j} (\text{mod}\;\; I)\otimes dz_j.
$$
Since $dp_i^2=0$, we see $d$ induces a map
$$
I/I^2\stackrel{D}\longrightarrow R\otimes_{P[n]}\Omega^1(P[n]).
$$
We wish to define the projection from $R\otimes_{P[n]}\Omega^1(P[n])$ to
$\Omega^1(R)$. But what is the $R$-module $\Omega^1(R)$ of
first differentials (officially called {\em Kaehler differentials})
for $R$, when
the corresponding variety may have singularities? The "quick-and-dirty" way,
for our expository purpose, is just to
define $\Omega^1(R)$ to be the cokernel of $D$ (see~\cite[p. 180]{Ma} for a formal definition).
In accordance, we have thus the natural projection
\begin{equation}\label{cokernel}
R\otimes_{P[n]}\Omega^1(P[n])\stackrel{\pi}\longrightarrow\Omega^1(R)
\end{equation}
given by
\begin{equation}\label{naturalproj}
1\otimes dz_j\longmapsto d(z_j+I):= dz_j+(\text{mod}\;\;dp_0,\cdots,dp_m).
\end{equation}
Hence we obtain
\begin{equation}\label{rightexact}
I/I^2\stackrel{D}\longrightarrow R\otimes_{P[n]}\Omega^1(P[n])\stackrel{\pi}\longrightarrow\Omega^1(R)\longrightarrow 0.
\end{equation}
The sequence cannot be made left exact in general:

\begin{example} Consider $I=(x^2,xy)$. We know $x^3\in I$ and
$$
D(x^3+I^2)=3x^2(\text{mod}\;\;I)\otimes dx=0.
$$
However, it can be easily checked that $x^3\notin I^2$.
\end{example}
The striking fact is that~\eqref{rightexact} can be made exact if we localize,
as in~\eqref{exact}, when we replace $P[n]$ by $R_{\mathcal P}$, 
$I$ by the maximal
ideal ${\mathcal P}_{\mathcal P}$ of $R_{\mathcal P}$, and $R=P[n]/I$ by
$R_{\mathcal P}/{\mathcal P}_{\mathcal P}=\kappa({\mathcal P})$, the quotient field
of the domain $R/{\mathcal P}$:
\begin{equation}\label{exquisite}
0 \longrightarrow {\mathcal P}_{\mathcal P}/({\mathcal P}_{\mathcal P})^2\stackrel{D}
\longrightarrow\kappa({\mathcal P})\otimes_{R_{\mathcal P}}\Omega^1(R_{\mathcal P})
\stackrel{\pi}\longrightarrow\Omega^1(\kappa({\mathcal P}))\longrightarrow 0,
\end{equation}
considered as vector spaces over $\kappa({\mathcal P})$. Here,
\begin{equation}\label{localized}
\Omega^1(R_{\mathcal P}):=R_{\mathcal P}\otimes_{R}\Omega^1(R)
\end{equation}
given by
$$
d(r/s):= -\frac{r}{s^2}\otimes ds+\frac{1}{s}\otimes dr
$$
with $\Omega^1(R)$ defined in~\eqref{cokernel}. (In fact, the equality in~\eqref{localized} can be derived as a consequence of the formal definition of 
Kaehler differentials~\cite[p. 187]{Ma}. We introduce it as a definition for the sake of expository convenience.)  

The underlying idea for the validity of~\eqref{exquisite} is hidden
in~\eqref{exact}. Namely, as long as we have a left inverse
$$
D^{-1}:\kappa({\mathcal P})\otimes_{R_{\mathcal P}}\Omega^1(R_{\mathcal P})
\longrightarrow
{\mathcal P}_{\mathcal P}/({\mathcal P}_{\mathcal P})^2
$$
such that
$$
D^{-1}\circ D=id,
$$
then $D$ is injective. Accordingly, given $D^{-1}$, one can define a morphism
\begin{equation}\label{1storder}
\nabla:h\in R_{\mathcal P}\longmapsto D^{-1}(1\otimes dh)\in 
{\mathcal P}_{\mathcal P}/({\mathcal P}_{\mathcal P})^2.
\end{equation}
Intuitively, $\nabla$ picks up the first order term of the Taylor expansion of
$h$, which can be seen by looking at~\eqref{TOY}, where
$$
D^{-1}:1\otimes d(\frac{f}{g})|_{y=0}\longmapsto
\frac{g_0f_1-f_0g_1}{g_0^2}y\;\;(\text{modulo higher order terms}),
$$
so that
\begin{equation}\nonumber
\nabla:\frac{f}{g}\longmapsto
\frac{g_0f_1-f_0g_1}{g_0^2}y\;\;(\text{modulo higher order terms}),
\end{equation}
where the right hand side is exactly the first order term of $f/g$ when we expand
it as
$$
\frac{f}{g}(x,y)=\frac{f_0}{g_0}+\frac{g_0f_1-f_0g_1}{g_0^2}y+\cdots.
$$
With the intuitive interpretation in mind, it is clear that
\begin{equation}\label{PpP}
h-\nabla(h)=0\in {\mathcal P}_{\mathcal P}/({\mathcal P}_{\mathcal P})^2,
\quad h\in {\mathcal P}_{\mathcal P}.
\end{equation}
Returning to~\eqref{1storder}, therefore, the map
$$ 
\iota:R_{\mathcal P}\longrightarrow R_{\mathcal P}/({\mathcal P}_{\mathcal P})^2,\quad h\longmapsto h-\nabla(h)
$$
intuitively picks up the $0th$ order term of $h$. Moreover, since
$\iota({\mathcal P}_{\mathcal P})=0$ by~\eqref{PpP},
it follows that $\iota$ descends to a map
$$
\iota:R_{\mathcal P}/{\mathcal P}_{\mathcal P}\longrightarrow R_{\mathcal P}/({\mathcal P}_{\mathcal P})^2,\quad h\longmapsto h-\nabla(h).
$$
In other words, the exact sequence
\begin{equation}\label{split}
0\longrightarrow {\mathcal P}_{\mathcal P}/({\mathcal P}_{\mathcal P})^2
\longrightarrow R_{\mathcal P}/({\mathcal P}_{\mathcal P})^2
\longrightarrow R_{\mathcal P}/{\mathcal P}_{\mathcal P}\longrightarrow 0
\end{equation}
splits by $\iota$ as ${\mathbb C}$-algebras. Conversely, the splitting of the sequence establishes the existence of $D^{-1}$,
so that~\eqref{exquisite} is true. We refer the reader to~\cite[p. 204]{Ma} for a proof of~\eqref{split}.

We are now ready to see why $(\bullet)$ holds true. Indeed, it suffices to
verify, via~\eqref{exquisite},
that, as vector spaces over $\kappa({\mathcal P})$, the dimension
of $\kappa({\mathcal P})\otimes_{R_{\mathcal P}}\Omega^1(R_{\mathcal P})$
equals that of $\Omega^1(\kappa({\mathcal P}))$. Now, since $\kappa({\mathcal P})$
is the quotient field of the domain $R/{\mathcal P}$, or rather, the rational
function field of the underlying irreducible variety $W$, the Kaehler module
$\Omega^1(\kappa({\mathcal P}))$ must be
of the same dimension as that of $W$, which is $n-m-1$, by the fact that
$p_0,\cdots,p_m$
defining the variety $V$ form a regular sequence so that $V$ is of pure dimension
$n-m-1$. (See~\cite[p. 191]{Ma} for a formal proof.) On the other hand, by~\eqref{naturalproj}, the image of $D$ in~\eqref{rightexact} is of
dimension $m+1$, the generic rank of the
Jocobian matrix $J$ by assumption,
as a vector space over $\kappa({\mathcal P})$,
so that $\Omega^1(R)$, the cokernel of $D$, is of dimension $n-m-1$ as a vector
space over $\kappa({\mathcal P})$. Consequently, by~\eqref{localized}, the dimension of $\kappa({\mathcal P})\otimes_{R_{\mathcal P}} \Omega^1(R_{\mathcal P})$ is $n-m-1$.

We have thus arrived at Serre's criterion of reducedness. 

\subsection{Codimension 2 estimate and normality} We now turn to the second
question as to under what condition a reduced ideal $I$ generated by
$p_0,\cdots,p_m$ in $P[n]$ is prime. Clearly, a necessary
condition is that the variety $V$ defined by $I$ is connected. It
turns out that the remaining condition sufficient for the primeness of $I$
is the codimension 2 Jacobian condition~\cite[p. 462]{Ei}.

\begin{theorem}\label{Serre1} {\rm (Serre)} Let $I$ be the ideal generated by
a regular sequence $p_0,\cdots,p_m,m+1\leq n,$ in $P[n]$
that define a connected variety $V$. Let $J$ be the subvariety of\, $V$
consisting of all points of $V$ where the Jacobian matrix
$$
\partial (p_0,\cdots,p_m)/\partial(z_1,\cdots,z_n)
$$
is not of full rank $m+1$. Suppose the codimension of $J$ is $\geq 2$ in $V$.
Then $I$ is a prime ideal.
\end{theorem}

To outline the proof, note that $V$ is reduced by Theorem~\ref{Serre}. Let
$p\in P[n]/I$ be a non-zerodivisor. Then $p_0,\cdots,p_m,p$\, form a regular sequence,
so that the ideal $I^*$ generated by $p_0,\cdots,p_m,p$, in view of $(\dagger)$, is the intersection
of minimal primes $Q_1,\cdots,Q_t$ containing $I^*$,
$$
I^*=\cap_{j=1}^t Q_j,
$$
and the algebraic set $V^*$ defined by $I^*$ is of pure dimension $n-m-2$.
Put intrinsically, this says that the (principal) ideal $(p)$ generated by
$p$ in $R=P[n]/I$ is the intersection of minimal primes $P_j:=Q_j/I$ containing
$p$ in $R$:
\begin{equation}\label{(p)}
(p)=\cap_{j=1}^t P_j.
\end{equation}
For ease of notation, let us denote any of the prime ideals $P_j$ by $P$.

We claim that $P_P$ is also generated by a single element
by~\eqref{exquisite}. The proof proceeds in a way entirely similar to the one
given in the preceding section. First of all, $\Omega^1(\kappa(P))=n-m-2$ because
the variety $V^*$ is of pure dimension $n-m-2$. Moreover, the middle space in~\eqref{exquisite}, as a vector space
over $\kappa(P)$,
has the same dimension $n-m-1$ as in the case of reducedness, because
the codimension 2 condition and the fact that $P$ defines an variety of
codimension 1 imply that the image of $D$ in~\eqref{rightexact} is still of
dimension $m+1$; therefore, the
dimension of $P_P/(P_P)^2$ is of dimension 1 as a vector space
over $\kappa(P)$. This is equivalent to saying that the {\em minimum} number of generators
of $P_P$ is 1, which is a consequence of the fundamental Nakayama lemma 
whose proof we leave to~\cite[p. 105]{Ku}. The claim follows.

So now $P_P=(f)$ in $R_P$. It follows that any element $x\in R_P$ is of the form
$x=uf^n$ for some integer $n\geq 0$ and some unit $u\in R_P$, i.e.,
$f$ is a {\em local uniformizing parameter}
for $R_P$. Indeed, since the units of $R_P$
constitute $R_P\setminus P_P$, an element $x\in R_P$ is either a unit, in which case we are done,
or $x\in P_P=(f)$, in which case $x=ff_1$ for some $f_1\in R_P$. Either $f_1$
is a unit and we are done, or $f_1=ff_2$ for some $f_2\in R_P$ with $x=f^2f_2$,
etc. It follows that
we have an ascending chain of ideals
$$
(f_1)\subset(f_2)\subset(f_3)\cdots,
$$
so that it must stabilize at some smallest $n$ (the {\em Noetherian} condition; a ring with the condition is called a {\em Notherian} ring). We obtain $x=f^nu$ for some unit $u$.
With this there comes the following simple but important observation.

\begin{proposition} Let $Q(R_P)$ be the quotient ring of $R_P$. Suppose
$a/b\in Q(R_P)$
satisfies a monic polynomial
\begin{equation}\label{monic}
t^k+c_{k-1}t^{k-1}+\cdots+c_1t+c_0
\end{equation}
in $t$, where $c_0,\cdots,c_{k-1}\in R_P$. Then $a/b\in R_P$.
\end{proposition}

To see this, write $a=f^lu$ and $b=f^mv$ for some units $u,v\in R_P$.
If $a/b\notin R_P$, then we have
$m>l$ so that $a/b=w/f^s$ with $s>0$ and $w$ a unit, which we
substitute into~\eqref{monic} to obtain
$$
(w/f^s)^k+c_{k-1}(w/f^s)^{k-1}+\cdots+c_1(w/f^s)+c_0=0;
$$
multiplying both sides by $f^{sk}$ we derive
$w^k=fg$ for some $g\in R_P$. This forces
$f\in P_P$ to be a unit, which is a contradiction. So, we conclude that
$a/b\in R_P$.

\begin{corollary}\label{R} It follows that $R$ satisfies the same property, namely,
that if $q/p\in Q(R)$, the quotient ring of $R$, satisfies a monic polynomial with coefficients
in $R$, then $q/p\in R$.
\end{corollary}

Indeed, if $q/p\in Q(R)\setminus R$ for a non-zerodivisor $p$,
then $q\notin(p)$. Following~\eqref{(p)}, we see $q\notin P_j$ for some j; call it $P$ for convenience.
It follows that $q\notin P_P$ since $q$ is a unit. Consequently,
$q/p\notin R_P$; for otherwise $q/p=a/b$ implies $q=pa/b\in P_P$, a contradiction.
But then $q/p$ satisfies a monic polynomial in $Q(R_P)$, which is induced by
the polynomial that $q/p$ satisfies in $Q(R)$; therefore, by the preceding proposition
$q/p\in R_P$, a contradiction.

\vspace{2mm}

In accordance with the corollary, we make the following definition.

\begin{definition} A reduced commutative ring $R$ with identity is {\em normal}
if whenever $x\in Q(R)$ satisfies a monic polynomial with coefficients in $R$,
there follows $x\in R$.
\end{definition}

Then Serre's criterion of primeness is a consequence of the following:

\begin{theorem}\label{normality} A Notherian normal ring $R$ is a direct product of normal domains. 
\end{theorem}

To see that Theorem~\ref{normality} implies Serre's criterion of normality, note first that the ring $R=P[n]/I$ under consideration
is normal by Corollary~\ref{R}. Thus Theorem~\ref{normality} concludes
that the variety $V$ defined by $I$ is a disjoint union of
irreducible varieties,
so that $V$ must be irreducible itself because it is connected. In other words,
$I$ is a prime ideal.

On the other hand, Theorem~\ref{normality} is a standard exercise in commutative algebra.
We refer the reader to~\cite[pp. 85-86]{Ku} for a proof.

Alternatively, we can understand normality from the function-theoretic point of view.
Recall that a function
$f$ is {\em weakly holomorphic} in an open set $O$ of $V$ if it is holomorphic on $O\setminus{\mathcal S}$
and is locally bounded in $O$. Passing to the limit as $O$ shrinks to a point $p$,
we can talk about the germs of weakly holomorphic functions at $p$. The variety
is said to be {\em normal} at $p$ if the germs of weakly holomorphic functions at $p$ coincide
with the germs of holomorphic functions at $p$. That is, the Riemann extension
theorem holds true in the germs of neighborhoods around $p$. $V$ is said to be
{\em normal} if it is normal at all its points.

If $V$ is normal, then its irreducible components are
disconnected; or else
a constant function with different values on different local irreducible branches,
which is not even continuous, would give
rise to a weakly holomorphic function that could be extended to a holomorphic function,
a piece of absurdity. This is the geometric meaning of Theorem~\ref{normality}.
See~\cite[p. 191]{Gu} for details.

\subsection{Algebraic independence of regular sequences} The 
Taylor expansion of~\eqref{taylor} can be viewed as follows. Let $I=(z_1,\cdots,z_k)$
be the ideal generated by the regular sequence $z_1,\cdots,z_k$.
In~\eqref{taylor}, we can think of
$$
f_0(x)\in P[n]/I, \quad f_1(x)y\in I/I^2,
\quad f_2(x)y^2\in I^2/I^3,\;\;\text{etc}.
$$
(Precisely, $y$ should be replaced by $y+I^2$.) On the other hand, we can also think of $z_{k+1},\cdots,z_n$ as generating $P[n]/I$,
so that $f_0(x),f_1(x),\cdots\in P[n]/I$. Hence, the polynomial
$f(x,y)\in P[n]$ written in~\eqref{taylor} can also be thought of
as a polynomial in $k$ formal variables $t_1,\cdots,t_k$ with coefficients in
$P[n]/I$, for which the expansion~\eqref{taylor} is the evaluation when we set
$t_1=z_1,t_2=z_2,\cdots,t_k=z_k\in I/I^2$. In other words, there is an isomorphism

\begin{eqnarray}\label{TAYLOR}
\aligned
&P[n]/I[t_1,\cdots,t_k]\longrightarrow P[n]/I\oplus I/I^2
\oplus I^2/I^3\oplus \cdots,\\
&t_i\longmapsto z_i+I^2,
\endaligned
\end{eqnarray}
where the left hand side is the polynomial ring with coefficients in $P[n]/I$
and the direct sum module on the right hand side consists of elements whose components
are zero eventually.

It turns out~\eqref{TAYLOR} is true for any regular sequence $p_1,\cdots,p_k\in P[n]$
and~\eqref{TAYLOR} continues to hold when we replace $z_1,\cdots,z_k$ by
$p_1,\cdots,p_k$, respectively, with the evaluation map
\begin{equation}\label{eval}
t_i\longmapsto p_i+I^2,1\leq i\leq k.
\end{equation}
Note that the evaluation in~\eqref{eval} is clearly
surjective. Since a polynomial is the sum of its homogeneous terms, the injectivity
of the evaluation comes down to proving the following:

\begin{proposition}\label{qp} Let $F(t_1,\cdots,t_k)$ be a homogeneous polynomial of
degree $d$ in $k$ variables $t_1,\cdots,t_k$ with coefficients in $P[n]$.
Suppose the evaluation results in $F(p_1,\cdots,p_k)\in I^{d+1}$. Then 
all the coefficients of $F$ belong to $I=(p_1,\cdots,p_k)$.
\end{proposition}

Since any homogeneous element $f\in I^{d+1}$ can be written as a
homogeneous $G(p_1,\cdots,p_k)$ of degree $d$ with coefficients in $I$, if we write
$F\in I^{d+1}$ as a sum of homogeneous terms $G_1,\cdots,G_m$ of degrees $\geq d+1$,
$$
F(p_1,\cdots,p_k)=G_1(p_1,\cdots,p_k)+\cdots+G_m(p_1,\cdots,p_k),
$$
and then regard each $G_j$ as a homogeneous polynomial of degree $d$ in
$t_1,\cdots,t_k$ with coefficients in $I$,
then
$$
H(t_1,\cdots,t_k):=F(t_1,\cdots,t_k)-G_1(t_1,\cdots,t_k)-\cdots-
G_m(t_1,\cdots,t_k)
$$
is homogeneous of degree $d$ with $H(p_1,\cdots,p_k)=0$. If we can establish that
this forces all the coefficients of $H$ to be in $I$, it will follow that
all the coefficients
of $F$ fall in $I$. Therefore, the above proposition is equivalent to:

\begin{proposition}\label{homog} Let $p_1,\cdots,p_k$ be a regular sequence in $P[n]$.
Let $F(t_1,\cdots,t_k)$ be a homogeneous polynomial of degree $d$
in $k$ variables with coefficients in $P[n]$.
Suppose $F(p_1,\cdots,p_k)=0$. Then 
all the coefficients of $F$ belong to $I=(p_1,\cdots,p_k)$.
\end{proposition}

We refer the reader to~\cite[p. 153]{Ku} for a short proof. Let us look at an important
application next.

\subsection{Method for generating regular sequences} Granted Serre's criteria
of reducedness and normality, checking that a sequence
$p_0,\cdots,p_m\in P[n]$ form a
regular sequence is by no means easy. The first condition of forming a regular
sequence is
that $p_0=\cdots=p_m=0$ defines a nonempty variety, or equivalently,
that $(p_0,\cdots,p_m)\neq P[n]$, which is already not that obvious to conclude.
However, if we now stipulate that $p_0,\cdots,p_m$ all be homogeneous of degree $\geq 1$, then automatically
$p_0=\cdots=p_m=0$ defines a connected and nonempty variety $V$, because 0
clearly belongs to $V$ and furthermore $V$ is connected since it is a
cone. Thus we can rephrase Serres's criterion of primeness in this case as follows:

\begin{theorem}\label{Serre2}\label{refined} Let $I$ be the ideal generated by
a regular sequence $p_0,\cdots,p_m,m+1\leq n,$ of homogeneous polynomials of degree $\geq 1$
in $P[n]$
that defines a variety $V$. Let $J$ be the subvariety of\, $V$
consisting of all points of\; $V$ where the Jacobian matrix
$$
\partial (p_0,\cdots,p_m)/\partial(z_1,\cdots,z_n)
$$
is not of full rank $m+1$. Suppose the codimension of $J$ is $\geq 2$ in $V$.
Then $I$ is a prime ideal.
\end{theorem}

From this we devised a criterion in~\cite{CCJ} and developed it further in~\cite{Ch1},~\cite{Ch3} to
construct regular
sequences in $P[n]$ that fits perfectly in the classification scheme of isoparametric
hypersurfaces.

\begin{lemma}\label{generatereg} Let $p_0,\cdots,p_m\in P[n]$ be linearly independent
homogeneous polynomials of equal degree $\geq 1$.
For each $0\leq k\leq m-1$, let $V_k$ be the variety defined by $p_0=\cdots=p_k=0$,
and let $J_k$ be the subvariety of\, $V_k$, where the Jacobian
$$
\partial(p_0,\cdots,p_k)/\partial(z_1,\cdots,z_n)
$$
is not of full rank $k+1$. If the codimension of $J_k$ in $V_k$ is $\geq 2$
for all $0\leq k\leq m-1$, then $p_0,\cdots,p_m$ form a regular sequence.
\end{lemma}

Indeed, Theorem~\ref{refined} applied to $p_0$ implies that $p_0$ is prime
and clearly $p_0$ forms a regular sequence. So, the statement $S(k)$ that the ideal
$I_k:=(p_0,\cdots,p_k)$ is prime and $p_0,\cdots,p_k$ form a regular sequence
holds for $k=0$.

Suppose the statement $S(k)$ holds. We show that
$p_{k+1}$ is not a zero divisor of $P[n]/I_k$. Let us assume
$$
p_{k+1}f=p_0f_0+p_1f_1+\cdots+p_kf_k
$$
for some $f,f_0,\cdots,f_k\in P[n]$. If $p_{k+1}$ vanishes entirely on $V_k$, then $p_{k+1}\in I_k$ by Nullstellensatz
as $I_k$ is a prime ideal. But then
$$
p_{k+1}=p_0g_0+\cdots+p_kg_k
$$
for some $g_0,\cdots,g_k\in P[n]$. However, since $p_0,\cdots,p_k,p_{k+1}$
are homogeneous of the same degree, we conclude that $g_0,\cdots,g_k$
are constants, which forces $p_0,\cdots,p_k,p_{k+1}$ to be linearly dependent.
This is a contradiction. Thus $p_{k+1}$ cannot vanish identically on $V_k$,
which implies that $f$ must vanish identically on $V_k$, so that $f\in I_k$.
Now that $p_{k+1}$ is not a zero divisor of $P[n]/I_k$, it follows that
$p_0,\cdots,p_{k+1}$ form a regular sequence, which, together with
the fact that $J_{k+1}$ is of codimension 2 in $V_{k+1}$, make
$I_{k+1}$ a prime ideal by Theorem~\ref{refined}, so that the statement $S(k+1)$ is true,
as long as $k\leq m-2$.

Lastly, when we reach that $I_{m-1}$ is prime, the
scheme results in the conclusion that $p_0,\cdots,p_m$ form a regular sequence.

\subsection{The syzygy of a regular sequence} Let $p_0,\cdots,p_m$
be a sequence in $P[n]$. The ideal
$$
Syz:=\{(q_0,\cdots,q_m):p_0q_0+\cdots+p_mq_m=0\}
$$
is called the {\em first syzygy ideal} of $p_0,\cdots,p_m$. Let $e_j:=(0\cdots,1,0\cdots)$,
where the only nonzero one $(=1)$ of the $m+1$ entries is at the $jth$ slot.
It is clear that $p_je_i-p_ie_j\in Syz$. $Syz$ is said to be {\em trivial}
if it is generated by $p_je_i-p_ie_j,i\neq j,$ in which case all $(q_0,\cdots,q_m)\in Syz$
are of the form

$$
q_a=\sum_{b=0}^m r_{ab}q_b,\quad r_{ab}=-r_{ba}.
$$

\begin{proposition}\label{syzy} The first syzygy ideal generated by a
regular sequence in $P[n]$ is trivial.
\end{proposition}

\begin{proof} Let $p_0,\cdots,p_m$ be a regular sequence. We do induction on $m$.

When $m=1$, given $p_0f_0+p_1f_1=0$, by the definition of a regular
sequence, we know $f_1=p_0h$ for some $h$. It follows that $f_0=-p_1h$.
The statement of the theorem is verified in this case.

Suppose the statement is true for $m=k$. For a regular sequence
$p_0,\cdots p_{k+1}$,
\begin{equation}\label{regu}
p_0f_0+\cdots+p_{k+1}f_{k+1}=0
\end{equation}
implies
\begin{equation}\label{regul}
f_{k+1}=r_{k+1\, 0}\,p_0+\cdots+r_{k+1\, k}\,p_k
\end{equation}
by the definition of a regular sequence. Substituting~\eqref{regul} into~\eqref{regu}
we obtain
$$
p_0(f_0+p_{k+1}\,r_{k+1\, 0})+\cdots+p_k(f_k+p_{k+1}\,r_{k+1\, k})=0.
$$
The induction hypothesis then ensures that
$$
f_a+p_{k+1}\,r_{k+1\, a}=\sum_{b=0}^k r_{ab}\,p_b,\quad r_{ab}=-r_{ba},\quad 0\leq a\leq k.
$$
That is,
$$
f_a=\sum_{b=0}^{k+1}r_{ab}\,p_b,\quad r_{ab}=-r_{ba},\quad 0\leq a\leq k+1,
$$
where we define $r_{a\, k+1}:=-r_{k+1\, a}, 0\leq a\leq k,$ with the latter
defined in~\eqref{regul}.
\end{proof}

\section{How the ideal theory interacts with isoparametric hypersurfaces}\label{ISO}
Through M\"{u}nzner's work~\cite[II]{Mu}, we know the number $g$
of principal curvatures of an isoparametric hypersurface $M$ in the sphere
is 1,2,3,4 or 6, and there are at most two multiplicities
$\{m_1,m_2\}$ of the principal curvatures of $M$, occurring alternately when
the principal curvatures are ordered,
where $m_1=m_2$ if $g$ is odd.
Over the ambient Euclidean space in which $M$ sits there is
a homogeneous polynomial $F$, called the Cartan-M\"{u}nzner polynomial,
of degree $g$ that satisfies
\begin{equation}\nonumber
|\nabla F|^2(x)=g^2|x|^{2g-2},\quad (\Delta F)(x)=(m_2-m_1)g^2|x|^{g-2}/2
\end{equation}
whose restriction $f$ to the sphere has image in $[-1,1]$ with $\pm 1$ the only critical values~\cite[I]{Mu}.
For any $c\in(-1,1)$, the preimage $f^{-1}(c)$ is an isoparametric
hypersurface with $f^{-1}(0)=M$. This 1-parameter of
isoparametric hypersurfaces degenerates to the two submanifolds
$f^{-1}(\pm 1)$ of codimension $m_1+1$ and $m_2+1$ in the sphere.

The isoparametric hypersurfaces with $g=1,2,3$ were classified by
Cartan to be homogeneous
~\cite{Car2},~\cite{Car3}. For $g=6$, it is known that $m_1=m_2=1$
or 2 by Abresch~\cite{A}. Dorfmeister
and Neher~\cite{DN1} showed that the isoparametric hypersurface is homogeneous in the
former case and Miyaoka~\cite{Mi},~\cite{Mi1} settled the latter.

For $g=4$, there are infinite classes of inhomogeneous examples
of isoparametric hypersurfaces, two of which were first constructed by Ozeki and Tackeuchi~\cite[I]{OT} 
to be generalized later by Ferus, Karcher and M\"{u}nzner~\cite{FKM},
referred to collectively as isoparametric hypersurfaces of OT-FKM type subsequently. We remark that
the OT-FKM type includes all the homogeneous examples barring the two
with multiplicities $\{2,2\}$ and $\{4,5\}$. To construct the OT-FKM type,
let $P_{0},\cdots,P_{m}$ be a Clifford system on
${\mathbb R}^{2l}$, which are orthogonal symmetric operators on
${\mathbb R}^{2l}$ satisfying
$$
P_{i}P_{j}+P_{j}P_{i}=2\delta_{ij}I,\;\;i,j=0,\cdots,m.
$$
The ${\rm 4}$th degree homogeneous
polynomial
$$
F(x)=|x|^{4}-2\sum_{i=0}^{m}(\langle P_{i}(x),x\rangle)^2
$$
is the Cartan-M\"{u}nzner polynomial,
where the angle brackets on the right hand side denote the Euclidean
inner product. 
The two multiplicities of the OT-FKM type
are $m$ and $k\delta(m)-1$ for any $k=1,2,3,\cdots$, where $\delta(m)$ is the
dimension of an irreducible module of the Clifford algebra $C_{m-1}$
with $l=k\delta(m)$. Stolz~\cite{St}
showed that these multiplicity pairs and $\{2,2\}$ and $\{4,5\}$
are exactly the possible multiplicities of isoparametric hypersurfaces
with four principal curvatures in the sphere.

To fix notation, we make the convention, by changing $F$ to $-F$
if necessary, that its two focal manifolds are $M_{+}:=F^{-1}(1)$ and
$M_{-}:=F^{-1}(-1)$ with respective codimensions $m_1+1\leq m_2+1$ in
the ambient sphere $S^{2(m_1+m_1)+1}$. The principal curvatures
of the shape operator $S_n$ of $M_{+}$ (vs. $M_{-}$) with respect to any unit
normal $n$ 
are $0,1$ and $-1$, whose multiplicities are, respectively,
$m_1,m_2$ and $m_2$ (vs. $m_2,m_1$ and $m_1$).

The third fundamental form of $M_{+}$ is the symmetric tensor
$$
q(X,Y,Z):=(\nabla^{\perp}_X S)(Y,Z)/3
$$
where $\nabla^{\perp}$ is the normal connection. For a chosen normal frame
$n_0,\cdots,n_{m_1}$ write
$$
p_a(X,Y):=\langle S(X,Y),n_a\rangle,\quad q_a(X,Y,Z)=\langle q(X,Y,Z),n_a\rangle,\quad 0\leq a\leq m_1.
$$
The Cartan-M\"{u}nzner polynomial $F$ is related to $p_a$ and $q_a$ by
the expansion formula of Ozeki and Takeuchi~\cite[I, p. 523]{OT}

\begin{eqnarray}\label{ot}
\aligned
&F(tx+y+w)=t^4+(2|y|^2-6|w|^2)t^2+8(\sum_{i=0}^{m_1}p_{i}w_{i})t\\
&+|y|^4-6|y|^2|w|^2+|w|^4-2\sum_{i=0}^{m_1}p_{i}^2-8\sum_{i=0}^{m_1}q_{i}w_{i}
\\
&+2\sum_{i,j=0}^{m_1}\langle\nabla p_{i},\nabla p_{j}\rangle w_{i}w_{j}
\endaligned
\end{eqnarray}
where $w:=\sum_{i=0}^{m_1} w_i n_i$, $y$ is tangential to $M_{+}$ at $x$,
$p_i:=p_i(y,y)$, $q_i:=q_i(y,y,y)$ and $\nabla$ is the Euclidean gradient. Note that our definition of $q_i$
differs from that of Ozeki and Takeuchi by a sign. An entirely similar
formula holds when $m_1$ is replaced by $m_2$.

In the expansion formula, the components of the second and third fundamental forms
are intertwined in ten convoluted equations. The first three say that the
shape operator $S_n$ satisfies $(S_n)^3=S_n$ for any normal direction $n$,
which is agreeable with the fact that the eigenvalues of $S_n$ are $0,1,-1$ with
fixed multiplicities. Set
$$
<p_a,q_b>:=\langle\nabla p_a,\nabla q_b\rangle,\quad 0\leq a, b\leq m_1.
$$
The fourth and fifth combined and the sixth are
\begin{eqnarray}\nonumber
\aligned
&<p_a,q_b>+<p_b,q_a>=0,\\
&<<p_a,p_b>,q_c>+<<p_c,p_a>,q_b>+<<p_b,p_c>,q_a>=0,\quad a,b,c\;\text{distinct}.
\endaligned
\end{eqnarray}
The seventh is
\begin{equation}\label{syzygy}
p_0q_0+\cdots+p_{m_1}q_{m_1}=0.
\end{equation}
Set $G:=\sum_{a=0}^{m_1}(p_a)^2.$ The last three are
\begin{eqnarray}\label{syzygy1}
\aligned
&16\sum_{a=0}^{m_1} (q_a)^2=16G\,|y|^2-<G,G>,\\
&8<q_a,q_a>\\
&=8(<p_a,p_a>|y|^2-(p_a)^2)+<<p_a,p_a>,G>-24G\\
&-2\sum_{b=0}^{m_1}<p_a,p_b>^2,\\
&8<q_a,q_b>\\
&=8(<p_a,p_b>|y|^2-p_a\,p_b)+<<p_a,p_b>,G>\\
&-2\sum_{c=0}^{m_1}<p_a,p_c><p_b,p_c>,\quad a,b\;\text{distinct}.
\endaligned
\end{eqnarray}

It looks at the first glance that it is a rather daunting task to tackle the classification
of isoparametric hypersurfaces with four principal curvatures in the sphere.
However,~\eqref{syzygy}, which appears to be the simplest of all the above
equations, brings good tidings.

Let us bring Proposition~\ref{syzy} into perspective.
Suppose now the components $p_0,\cdots,p_{m_1}$ of the second fundamental
form constitute a regular sequence. Then Proposition~\ref{syzy} warrants that the components
$q_0,\cdots,q_m$ of the third fundamental form satisfy
\begin{equation}\label{eq4}
q_a=\sum_{b=0}^{m_1} r_{ab}\,p_b,
\end{equation}
where $r_{ab}=-r_{ba}$ are homogeneous of degree 1.

Now let us introduce the Euclidean coordinates of the eigenspaces $V_{+},V_{-},V_0$, with eigenvalues
1, -1, 0, respectively, of the shape operator $S_{n_0}$ to be
\begin{eqnarray}\nonumber
\aligned
&z_p,\quad m_1+1\leq p\leq 2m_1,\\
&u_\alpha,\quad 2m_1+1\leq\alpha\leq 2m_1+m_2,\\
&v_\mu,\quad 2m_1+m_2+1\leq\mu\leq 2m_1+2m_2,
\endaligned
\end{eqnarray}
with respect to which we write

\begin{equation}\label{eq5}
r_{ab}:=\sum_{\alpha} T_{ab}^{\alpha}u_\alpha+\sum_{\mu}T_{ab}^{\mu}v_\mu
+\sum_{p}T_{ab}^pz_p.
\end{equation}
We have

\begin{eqnarray}\label{eqnarray}\label{eq6}
\aligned
p_0&=\sum_{\alpha}(u_\alpha)^2-\sum_{\mu}(v_\mu)^2,\\
p_a&=2\sum_{\alpha\mu}S^a_{\alpha\mu}u_\alpha v_\mu
+2\sum_{\alpha p}S^a_{\alpha p}u_\alpha z_p
+2\sum_{\mu p}S^a_{\mu p}v_{\mu}z_p,
\endaligned
\end{eqnarray}
for $1\leq a\leq m_1$, where we set
$$
S^a_{\alpha\mu}:=\langle S(X_\alpha,Y_\mu),n_a\rangle,
$$
etc.,
with $X_\alpha,Y_\mu,$ and $Z_p$ the orthonormal bases for the coordinates $u_\alpha,v_\mu,$
and $w_p$, respectively. We claim that

\begin{equation}\label{eq7}
T^\alpha_{a0}=T^\mu_{a0}=0,
\end{equation}
for $1\leq a\leq m_1$. To this end, we calculate $q_a$ in two ways. On the one hand,
substituting~\eqref{eq5} and~\eqref{eq6} into~\eqref{eq4}, we see that $q_a$
has the term 
$$
(\sum_\alpha T^\alpha_{a0}u_\alpha)(\sum_\beta (u_\beta)^2)+\cdots,
$$
so that the coefficient of $(u_\alpha)^3$ in $q_a$, denoted by $q_a^{\alpha\alpha\alpha}$, is
$$
q_a^{\alpha\alpha\alpha}=T^\alpha_{a0}.
$$
On the other hand, by a direct inspection, the right hand side of the first
identity of~\eqref{syzygy1} has no $(u_\alpha)^6$-term,
so that $q_a^{\alpha\alpha\alpha}=0$.

Next, we calculate $q_0$ in two ways. On the one hand,
we expand $q_0$
by~\eqref{eq4},~\eqref{eq5},~\eqref{eq6}, and~\eqref{eq7}, keeping in mind
that $q_0$ is homogeneous of degree 1 in $u_\alpha,v_\mu$ and $z_p$,
by~\cite[I, p. 537]{OT}, to obtain that the coefficient of the
$u_\alpha v_\mu z_p$-term
of $q_0$, denoted by $q_0^{\alpha\mu p}$, is
\begin{equation}\label{eq8}
q_0^{\alpha\mu p}=2\sum_{b\geq 1}T^p_{0b}S^b_{\alpha\mu}.
\end{equation}                     
On the other hand, traversing along the great circle spanned by $x$ and $n_0$ by length
$\pi/2$,
we end up again on $M_{+}$ at $n_0$ with $x$ as a normal vector.
Accordingly, set $x^{\#}:=n_0\in M_{+}$ and ${\bf n}^{\#}_0:=x$ normal to $M_{+}$ at
$x^{\#}$.

At $x^{\#}$, set 

\begin{eqnarray}\nonumber
\aligned
&t^{\#}=w_0,\quad u_1^{\#}=u_1,\cdots,u_{m_2}^{\#}=u_{m_2},\quad v_1^{\#}=v_1,
\cdots, v_{m_2}^{\#}=v_{m_2},\\
&z_1^{\#}=w_1,\cdots,z_{m_1}^{\#}=w_{m_1},\quad
w_0^{\#}=t,w_1^{\#}=z_1,\cdots,w_{m_1}^{\#}=z_{m_1}.
\endaligned
\end{eqnarray}
Then with $|y|^2=|u|^2+|v|^2+|z|^2$, 
it is easily checked that $F$ in~\eqref{ot} will be converted to
$$
(t^{\#})^4+(2|y^{\#}|^2-6|w^{\#}|^2)(t^{\#})^2+|y^{\#}|^4
-6|y^{\#}|^2|w^{\#}|^2+|w^{\#}|^4+\cdots.
$$
In other words, the eigenspaces $V_{+}^{\#},V_{-}^{\#},V_0^{\#}$ of
$S_{n^{\#}_0}$ with eigenvalues $1,-1,0$ are,
respectively,
$V_{+},V_{-},n_0^{\perp}:={\rm span}(n_1,\cdots,n_{m_1})$.
Moreover,
${\mathbb R}x\oplus V_0$ is the normal space to $M_{+}$ at $x^{\#}$. (See~\cite[p. 15]{CCJ}
for a geometric proof.)

Note that the third term of~\eqref{ot} at $x^{\#}$, which is
$$
8(\sum_{a=0}^{m_1}p_a^{\#}w_a^{\#})t^{\#},
$$
is what determines the second fundamental form $S^{\#}$ at $x^{\#}$; in fact,
only $-8q_0w_0$ of~\eqref{ot} at $x$, when substituted by the
$\#$-quantities, contributes to the
$u_\alpha v_\mu$-components of $S^{\#}$. So, expanding $-8q_0w_0$ in
$z_1,\cdots,z_{m_1}$, we obtain
\begin{equation}\label{eq9}
8q_0w_0=8(\sum_{p}H^pz_p)w_0=8(\sum_p H^pw_p^{\#})t^{\#},
\end{equation}
where
\begin{equation}\label{eq10}
H^p:=2\sum_{\alpha\mu} S^p_{\alpha\mu}u_\alpha v_\mu,
\end{equation}
and $S^p_{ij}$ denotes the tangential $(ij)$-component of the second fundamental form
of $M_{+}$ in the normal $p$-direction at $x^{\#}.$
Here, we invoke again the fact that $q_0$ is homogeneous of degree 1 in all
$x_{\alpha},y_{\mu},z_{p}$.

Comparing~\eqref{eq8},~\eqref{eq9} and~\eqref{eq10}
we derive
$$
S^p_{\alpha\mu}=\sum_b f^p_b S^b_{\alpha\mu},\quad
f^p_b=T^p_{0b}.
$$
Therefore, we may assume, with the index range $m_1+1\leq p\leq 2m_1$, that 
\begin{equation}\label{anti-sym1}
S^{a+m_1}_{\alpha\mu}=S^a_{\alpha\mu},
\end{equation}
by an orthonormal frame change, so long as we can show that the matrix $\begin{pmatrix}f^p_b\end{pmatrix}$ is orthogonal.
Remarkably, this is indeed true! The key is the second identity of~\eqref{syzygy1}, where
we can employ the commutative algebra scheme
Proposition~\ref{homog} to rewrite it as a polynomial homogeneous
in all $p_ap_b$ whose coefficients are homogeneous polynomials of degree 2,
so that these coefficients are linear combinations of all $p_a$.
Specifically, the coefficient of $(p_0)^2$ is
$$
16\sum_{a=1}^{m_1}(r_{0a})^2-16(\sum_\alpha(u_\alpha)^2+\sum_\mu(v_\mu)^2
+\sum_p(z_p)^2)
+4<p_0,p_0>,
$$
which is a linear combination of $p_0,p_1,\cdots,p_{m_1}$. Knowing that $r_{0a}$ are
functions of $z_p$ alone by~\eqref{eq5} and~\eqref{eq7}, we invoke~\eqref{eq6} and compare variable
types to conclude
that
\begin{equation}\label{orthogonal}
\sum_{a=1}^{m_1}(r_{0a})^2=\sum_{p=m_1+1}^{2m_1}(z_p)^2.
\end{equation}
But then~\eqref{eq5} for $r_{0a}$ in terms of~\eqref{orthogonal} says
exactly that the matrix
$\begin{pmatrix}f^p_b\end{pmatrix}$ is orthogonal.

Now that
$$
f^{a+m_1}_b=\delta^a_b
$$
for~\eqref{anti-sym1} to hold, we deduce by~\eqref{eq5} and~\eqref{eq7}

$$
r_{0b}=\sum_{a}\delta^{a}_{b}z_{a+m_1}=z_{b+m_1},
$$
and, invoking the Einstein summation convention,
\begin{equation}\nonumber
\aligned
&q_0=r_{0b}\,p_b\\
&=2(\delta^{a}_{b}z_{a+m_1})
(S^b_{\alpha\mu}u_\alpha v_\mu+S^b_{\alpha\; c+m_1}u_\alpha z_{c+m_1}
+S^b_{\mu\; c+m_1}v_{\mu}z_{c+m_1}).
\endaligned
\end{equation}
Hence, we obtain
$$
\sum_{a b c\alpha}(\delta^{a}_{b}z_{a+m_1})(S^b_{\alpha\; c+m_1}u_\alpha z_{c+m_1})=0
$$
or equivalently, 
$$
\sum_{ac}S^a_{\alpha\, c+m_1}z_{c+m_1}z_{a+m_1}=0.
$$
In other words, we have
\begin{equation}\label{anti-sym2}
S^a_{\alpha\, c+m_1}=-S^c_{\alpha\, a+m_1}.
\end{equation}
Likewise, we have
\begin{equation}\label{anti-sym3}
S^a_{\mu\, c+m_1}=-S^c_{\mu\, a+m_1}.
\end{equation}

It is evident now that~\eqref{anti-sym1},~\eqref{anti-sym2}, and~\eqref{anti-sym3}
enjoy a certain "Clifford" property. In fact, as shown in~\cite{Ch}, the geometric meaning
of these three equations is that they give rise to intrinsic isometries on
$M_{+}$ that exactly form the $Spin$-action on $M_{+}$ in the case when the isoparametric
hypersurface is of OT-FKM type. Moreover, we showed in~\cite{Ch1}, based on~\cite{CCJ},~\cite{Ch},
that if we assume
the mild condition that $m_1<m_2$, which essentially says that $M_{+}$ is sufficiently curved, then these intrinsic isometries extend to
extrinsic isometries of the ambient sphere to yield the OT-FKM type:


\begin{proposition}\label{iso}
Let $m_1<m_2$. If~\eqref{anti-sym1},~\eqref{anti-sym2}, and~\eqref{anti-sym3} hold,
then the hypersurface is of OT-FKM type. In particular, if $m_1<m_2$ and
the components of the second fundamental form $p_0,p_1,\cdots,p_{m_1}$
of $M_{+}$ form a regular sequence,
then the isoparametric hypersurface is of OT-FKM type.
\end{proposition}

By this proposition, the classification of isoparametric hypersurfaces
with four principal curvatures now boils down to exploring
Lemma~\ref{generatereg} to warrant that the components
$p_0,\cdots,p_{m_1}$ of the second fundamental form of $M_{+}$ constitute a regular sequence.
To this end, let us look at the $p_0,\cdots,p_k,k\leq m_{1}-1$. Following Lemma~\ref{generatereg}
we must estimate the codimension of $J_k$ in $V_k$ by understanding the rank of
the Jacobian matrix of $p_0,\cdots,p_k$.

Let us parametrize ${\mathbb C}^{2m_2+m_1}$ by
points $(u,v,w)$
with coordinates $u_\alpha,v_\mu$, and $w_p$, where $1\leq\alpha,\mu\leq m_2$,
and $1\leq p\leq m_1$. For $0\leq k\leq m_1$, let
$$
V_k:=\{(u,v,w)\in{\mathbb C}^{2m_2+m_1}:p_0(u,v,w)=\cdots=p_k(u,v,w)=0\}.
$$
We first estimate the dimension of the subvariety $X_k$ of ${\mathbb C}^{2m_2+m_1}$, where
$$
X_k:=\{(u,v,w)\in{\mathbb C}^{2m_2+m_1}
:\text{rank of the Jacobian of}\; p_0,\cdots,p_k<k+1\}.
$$
This amounts to saying that $dp_0,\cdots,dp_k$ are linearly dependent, or, that there are
constants $c_0,\cdots,c_k$ such that

\begin{equation}\label{linear}
c_0dp_0+\cdots+c_kdp_k=0.
\end{equation}
Since $p_a=\langle S_a(x),x\rangle$, we see $dp_a=2\langle S_a(x),dx\rangle$
for $x=(u,v,w)^{tr}$;
therefore, by~\eqref{linear}
$$
X_k=\{(u,v,w):(c_0S_0+\cdots+c_kS_k)\cdot (u,v,w)^{tr}=0\}.
$$
for $[c_0:\cdots:c_k]\in{\mathbb C}P^k$,
Here, $\langle S_{a}(X),Y\rangle=\langle S(X,Y),n_a\rangle$ is the shape operator of the focal
manifold
$M_{+}$ in the normal direction $n_a$.
By Lemma~\ref{generatereg}, we wish to establish

\begin{equation}\nonumber
\dim(X_k\cap V_k)\leq\dim(V_k)-2
\end{equation}
for $k\leq m_1-1$ to verify that $p_0,p_1,\cdots,p_{m_1}$ form a
regular sequence since

\begin{equation}\label{prim}
J_k=X_k\cap V_k.
\end{equation}

Note that for a fixed
$\lambda=[c_0:\cdots:c_k]\in {\mathbb C}P^k$,
if we set
$$
{\mathscr S}_\lambda:=\{(u,v,w): (c_0S_0+\cdots+c_kS_k)\cdot (u,v,w)^{tr}=0\},
$$
then we have
\begin{equation}\label{unio}
X_k=\cup_{\lambda\in{\mathbb CP}^k} {\mathscr S}_\lambda.
\end{equation}
Thus, it is fundamental to estimate the dimension of ${\mathscr S}_\lambda$.

We break it into two cases.
If $c_0,\cdots,c_{k}$ are either all real or all purely imaginary, then
$$
\dim({\mathscr S}_\lambda)=m_1,
$$
since $c_0S_{n_0}+\cdots+c_kS_{n_k}=cS_n$ for some unit normal vector $n$ and
some nonzero real or purely imaginary constant $c$, and we know that the null space
of $S_n$ is of dimension $m_1$ for all normal $n$.

On the other hand, if $c_0,\cdots,c_k$ are not all
real and not all purely imaginary, then after a normal basis change,
we can assume that

\begin{equation}\label{slamb}
{\mathscr S}_\lambda=\{(u,v,w):(S_{1^*}-\mu_\lambda S_{0^*})
\cdot(u,v,w)^{tr}=0\}
\end{equation}
for some complex number $\mu_\lambda$ relative to a new orthonormal normal basis $n^*_0,n^*_1,\cdots,n^*_{k}$
in the linear span of $n_0,n_1,\cdots,n_k$; explicitly, $n_0^*$ and $n_1^*$
are obtained by decomposing $n:=c_0n_0+\cdots+c_kn_k$ into its real and imaginary
parts $n=\alpha+\sqrt{-1}\beta$ and define $n_0^*$ and $n_1^*$ by performing the Gram-Schmidt
process.

In matrix terms, the equation in~\eqref{slamb} assumes the form

\begin{equation}\label{Matrix}
\begin{pmatrix}0&A&B\\A^{tr}&0&C\\B^{tr}&C^{tr}&0\end{pmatrix}\begin{pmatrix}x\\y\\z\end{pmatrix}
=\mu_\lambda\begin{pmatrix}I&0&0\\0&-I&0\\0&0&0\end{pmatrix}\begin{pmatrix}x\\y\\z\end{pmatrix},
\end{equation}
where $x,y$, and $z$ are (complex) eigenvectors of $S_{0^*}$ with
eigenvalues $1,-1$, and $0$, respectively.

Suffices it to say, leaving the details to~\cite{Ch1}, that~\cite[Lemma 49, p. 64]{CCJ}
ensures that we can normalize the matrix on the left hand side
of~\eqref{Matrix} to
decompose $x,y,z$ into $x=(x_1,x_2),y=(y_1,y_2),z=(z_1,z_2)$
with $x_2,y_2,z_2\in{\mathbb C}^{r_\lambda}$, where $r_\lambda$ is the rank of $B$, or
intrinsically, $m_1-r_\lambda$ is the dimension of the intersection of the kernels of
$S_{0^*}$ and $S_{1^*}$. With respect to this decomposition 
either $x_1=y_1=0$ or both are nonzero
with $\mu_\lambda=\pm\sqrt{-1}$. In both cases we have
$x_2=-y_2$ and can be solved in $z_2$ 
so that $z$ can be chosen to be a free variable. Hence, either $x_1=y_1=0$, in which case

$$
\dim({\mathscr S}_\lambda)=m_1,
$$
or both $x_1$ and $y_1$ are nonzero,
in which case $y_1=\pm\sqrt{-1}x_1$ and so

\begin{equation}\label{eSt}
\dim({\mathscr S}_\lambda)=m_1+m_2-r_\lambda,
\end{equation}
where $x_1$ contributes dimension $m_2-r_\lambda$ while $z$ does $m_1$. Now since
by~\eqref{prim} and~\eqref{unio}

\begin{equation}\label{jk}
J_k=X_k\cap V_k=\cup_{\lambda\in{\mathbb C}P^k} ({\mathscr S}_\lambda\cap V_k),
\end{equation}
where $V_k$ is defined by $p_0=\cdots=p_k=0$ and also by $p_{0^*}=\cdots=p_{k^*}$,
let us cut ${\mathscr S}_\lambda$ by
$$
0=p_{0^*}=\sum_{\alpha}(x_\alpha)^2-\sum_{\mu}(y_\mu)^2
$$
to achieve an initial estimate of $\dim(J_k)$.

\vspace{2mm}

\noindent Case 1: $x_1$ and $y_1$ are both nonzero. This is the case of
nongeneric $\lambda\in{\mathbb C}P^k$. We substitute $y_1=\pm \sqrt{-1}x_1$ and $x_2$
and $y_2$ in terms of $z_2$ into $p_{0^*}=0$ to
deduce that
$$
0=p_{0^*}=(x_1)^2+\cdots+(x_{m_2-r_\lambda})^2+z\; \text{terms};
$$
hence, $p_{0^*}=0$ cuts ${\mathscr S}_\lambda$ to reduce the dimension by 1.
That is, now by~\eqref{eSt},

\begin{equation}\label{Sub}
\dim(V_{k}\cap{\mathscr S}_\lambda)\leq (m_1+m_2-r_\lambda)-1\leq m_1+m_2-1,
\end{equation}
noting that $V_k$ is also cut out by $p_{0^*},p_{1^*},\cdots,p_{k^*}$. Meanwhile,
only a subvariety of $\lambda$ of
dimension $k-1$ in ${\mathbb C}P^k$ assumes $\mu_\lambda=\pm\sqrt{-1}$;
in fact, this subvariety
is a smooth hyperquadric ${\mathcal Q}_{k-1}$ in ${\mathbb C}P^k$. This is because
if we write $(c_0,\cdots,c_k)=\alpha+\sqrt{-1}\beta$ where
$\alpha$ and $\beta$ are
real vectors, then $\mu_\lambda=\pm\sqrt{-1}$
is equivalent to the conditions that
$\langle\alpha,\beta\rangle=0$ and $|\alpha|^2=|\beta|^2$. That is,
the nongeneric
$\lambda\in{\mathbb C}P^k$ constitute the smooth hyperquadric.
Therefore, by~\eqref{jk}, an irreducible component ${\mathcal W}$
of $J_k$ over nongeneric $\lambda$
will satisfy

\begin{equation}\nonumber
\dim({\mathcal W})\leq\dim(V_{k}\cap{\mathscr S}_\lambda)+k-1\leq
m_1+m_2+k-2.
\end{equation}
(Total dimension $\leq$ base dimension $+$ fiber dimension.)

\vspace{2mm}
 
\noindent Case 2: $x_1=y_1=0$. This is the case of generic $\lambda$, where
$\dim({\mathscr S}_\lambda)=m_1$, so that an irreducible component
${\mathcal V}$
of $J_k$ over generic $\lambda$ will satisfy

$$
\dim({\mathcal V})\leq m_1+k\leq m_1+m_2+k-2,
$$
as we may assume that $m_2\geq 2$, noting that
the case $m_1=m_2=1$ is straightforward~\cite[p. 61]{CCJ}.

Putting these two cases together, we conclude that

\begin{equation}\label{Crucial}
\dim(J_k)=\dim(X_k\cap V_k)\leq m_1+m_2+k-2.
\end{equation}
On the other hand, since $V_k$ is cut out by $k+1$ equations
$p_0=\cdots=p_k=0$, we have

\begin{equation}\label{Lower-bound}
\dim(V_k)\geq m_1+2m_2-k-1.
\end{equation}
Therefore,
\begin{equation}\label{cod2}
\dim(J_k)\leq \dim(V_k)-2
\end{equation}
when $k\leq m_1-1$, taking $m_2\geq 2m_1-1$ into account.

In summary, we have established~\eqref{cod2} for $k\leq m_1-1$,
so that the ideal $(p_0,p_1,\cdots,p_k)$ is prime when $k\leq m_1-1$.
Lemma~\ref{generatereg} 
then implies that $p_0,p_1,\cdots,p_{m_1}$ form a regular sequence.
It follows by Proposition~\ref{iso} 
that the isoparametric hypersurface is of OT-FKM type. Thus, we derived
in~\cite{Ch1} the classification proven in~\cite{CCJ} in a simpler fashion:

\begin{theorem}\label{normalcool} Assume $m_2\geq 2m_1-1$. Then the isoparametric hypersurface
with four principal curvatures is of OT-FKM type.
\end{theorem}

By the multiplicity result of Stolz~\cite{St}, which says 
that $(m_1,m_2)$ is either $(2,2),(4,5)$ or that of an isoparametric hypersurface
of OT-FKM type, Theorem~\ref{normalcool} finishes off all the isoparametric hypersurfaces with four principal curvatures,
except when $(m_1,m_2)=(3,4),(4,5),(6,9)$ or $(7,8)$. The class of 
isoparametric hypersurfaces in the theorem are tied with {\em complete intersections}, i.e., those polynomial ideals generated
by regular sequences.
In sharp contrast, the four remaining cases have the peculiar property,
due to the fact that they are tied with quaternion and octonion algebras, that
$p_0,\cdots,p_{m_1}$ fail to be regular sequences; for if they formed
a regular sequence, Proposition~\ref{iso} would imply that the isoparametric hypersurface
was to be of OT-FKM type where the Clifford action acted on $M_{+}$. However,
such an isoparametric hypersurface can never be of OT-FKM type when $(m_1,m_2)=(4,5)$,
whereas for $(m_1,m_2)=(3,4), (6,9)$ or $(7,8)$, there
are examples in the same ambient sphere where the Clifford action acts on $M_{-}$.
This is a contradiction. Thus, $p_0,\cdots,p_{m_1}$ cannot be regular.
Irregular sequences, even over complex numbers, can be wildly untamed.

It turns out that {\em Condition A} of Ozeki and Takeuchi plays a
decisive role in handling the exceptional cases when the multiplicity pair
is $(m_1,m_2)=(3,4),(4,5)$ or $(6,9)$.

\begin{definition} A point $p\in M_{+}$ is of Condition A if
$S_n$ at $p$ share the same kernel in all normal directions $n$.
\end{definition}

Originally Ozeki and Takeuchi~\cite[I]{OT} introduced Conditions A and B in their construction
of two families of inhomogeneous isoparametric hypersurfaces with four principal curvatures
with multiplicity pair $(m_1,m_2)=(3,4k)$ or $(7,8k)$ for $k\geq 1$, where the Clifford
action acts on $M_{+}$. Later, Dorfmeister
and Neher~\cite{DN} showed that Condition A alone implies that the isoparametric
hypersurface is of OT-FKM type (see also~\cite{Ch2}); in particular, in the case when
$(m_1,m_2)=(3,4)$ or $(7,8)$, either the Clifford action acts on $M_{+}$,
which are the ones constructed by Ozeki and Takeuchi, or it acts on $M_{-}$,
which are the ones constructed by Ferus, Karcher and M\"{u}nzner.
In contrast, those isoparametric hypersurfaces
with $(m_1,m_2)=(4,5)$ or $(6,9)$ do not admit any points of Condition A.

We must now come up with a finer estimate on the right hand side of~\eqref{Sub}
in which the quantity $r_\lambda$ is entirely discarded. This is where Condition A
comes in.

Note that if we stratify the above hyperquadric ${\mathcal Q}_{k-1}$ of nongeneric
$\lambda\in{\mathbb C}P^k$
into subvarieties ${\mathcal L}_j$ over which
$r_\lambda=j$, then by~\eqref{Sub} an irreducible component ${\mathcal W}$
of $V_k\cap (\cup_{\lambda\in {\mathcal L}_j} {\mathscr S}_\lambda)$
will satisfy                            
\begin{equation}\nonumber
\dim({\mathcal W})\leq\dim(V_{k}\cap{\mathscr S}_\lambda)+k-1\leq
m_1+m_2+k-2-j.
\end{equation}


We run through the same arguments as that following~\eqref{Sub} to deduce that
the codimension 2 estimate~\eqref{cod2} holds true over ${\mathcal L}_j$ when

\begin{equation}\label{refinedest}
m_2\geq 2k+1-j.
\end{equation}

Let us look at the case when $(m_1,m_2)=(3,4)$. Here, $0\leq k\leq m_1-1=2$. First observe that~\eqref{refinedest} is automatically
satisfied when $j\geq 1$. Assume $j=0$ now; let $\lambda_0$ be an element in ${\mathcal L}_0$ and so $r_{\lambda_0}=j=0$.  

Suppose that $M_{+}$ is free of points of Condition A everywhere. 
Since $r_{\lambda_0}=0$, the matrices $B=C=0$ and $A=I$ in~\eqref{Matrix} for $S_{1^*}$. For notational clarity,
let us denote the associated $B$ and $C$ blocks of the shape operator matrices $S_{n_a^*}$ by $B_{a^*}$ and $C_{a^*}$ for the normal basis elements
$n_1^*,
\cdots,n_{m_1}^*$.
It follows that $p_{0^*}=0$ and $p_{1^*}=0$ cut ${\mathscr S}_{\lambda_0}$ in the variety

\begin{equation}\nonumber
\{(x,\pm \sqrt{-1}x,z):\sum_\alpha (x_\alpha)^2=0\}.
\end{equation}
$(B_{2^*},C_{2^*})$ or $(B_{3^*},C_{3^*})$ must be nonzero since $M_{+}$ has no points of Condition A; assume it is the former. 
Since $z$ is a free variable, $p_{2^*}=0$ will have nontrivial $z$-terms

\begin{eqnarray}\nonumber
\aligned
0=p_{2^*}&=\sum_{\alpha p}S_{\alpha p}x_\alpha z_p+\sum_{\mu p}T_{\mu p}y_\mu z_p+x_\alpha y_\mu\;\text{terms}\\
&=\sum_{\alpha p}(S_{\alpha p}\pm\sqrt{-1}T_{\alpha p})x_\alpha z_p+x_\alpha x_\mu\;\text{terms},
\endaligned
\end{eqnarray}
taking $y=\pm\sqrt{-1}x$ into account, where $S_{\alpha p}:=\langle S(X^*_\alpha,Z^*_p),n^*_2\rangle$
and $T_{\mu p}:=\langle S(Y^*_\mu,Z^*_p),n^*_2\rangle$ are (real) entries of $B_{2^*}$ and $C_{2^*}$,
respectively, and $X^*_{\alpha},1\leq\alpha\leq m_2$,
$Y^*_{\mu},1\leq\mu\leq m_2$, and $Z^*_{p},1\leq p\leq m_1$, are orthonormal
eigenvectors for the eigenspaces of $S_{n^*_0}$ with eigenvalues $1,-1,$ and $0$,
respectively;
hence, the dimension
of ${\mathscr S}_{\lambda_0}$ will be cut down by 2 by $p_{0^*},p_{1^*},p_{2^*}=0$. In conclusion, modifying~\eqref{Sub} we have

\begin{equation}\nonumber
\dim(V_{2}\cap{\mathscr S}_\lambda)\leq m_1+m_2-2,
\end{equation} 
for all $\lambda\in{\mathcal L}_0$. As a consequence,
the right hand side of~\eqref{refinedest}, which is 5 for $j=0$,
is now cut down to 4 with the additional $p_{2^*}=0$
so that the codimension 2 estimate goes through for ${\mathcal L}_0$ as well. It follows that the isoparametric hypersurface is in fact the example 
constructed by Ozeki and Takeuchi of OT-FKM type, which thus has points of Condition A, a contradiction to the assumption that $M_{+}$ has no points of 
Condition A. Therefore, $M_{+}$ admits points of Condition A.
But then the result of Dorfmeister and Neher implies the isoparametric hypersurface is of OT-FKM type~\cite{Ch1}:

\begin{theorem} Let $(m_1,m_2)=(3,4)$. Then the isoparametric hypersurface is either the homogeneous one, or is the inhomogeneous one constructed by Ozeki and Takeuchi.
\end{theorem}

For $(m_1,m_2)=(4,5)$ (vs. $(m_1,m_2)=(6,9)$) and $0\leq k\leq m_{1}-1=3$ (vs. $0\leq k\leq m_1-1=5$), {\em a priori}~\eqref{refinedest} gives 
$5\geq 7-j\geq 2k+1-j$ (vs. $9\geq 11-j\geq 2k+1-j$). Therefore, the codimension 2 estimate goes through for $j\geq 2$ in both cases. Thus it 
looks hopeful that one will only have to handle $j\leq 1$ for the classification. Indeed, this is so. Employing the fact that $M_{+}$ admits
no points of Condition A in the case of these two multiplicity pairs, a delicate analysis was performed in~\cite{Ch3} to establish that either the isoparametric hypersurface
is the inhomogeneous one constructed by Ferus, Karcher and M\"{u}nzner in the $(6,9)$ case where the Clifford action acts on $M_{+}$, or the second fundamental 
form of $M_{+}$ is exactly that of the homogeneous example in either case. The classification result follows by pinning down the third fundamental form to determine
uniquely the Cartan-M\"{u}nzner polynomial via the expansion formula of Ozeki and Takeuchi, where~\eqref{syzygy1} plays a decisive role~\cite{Ch3}:

\begin{theorem} Let $(m_1,m_2)=(4,5)$ or $(6,9)$. Then the isoparametric hypersurface with four principal curvatures is either homogeneous, or is the inhomogeneous 
one constructed by Ferus, Karcher and M\"{u}nzner in the latter case.
\end{theorem}

$(m_1,m_2)=(7,8)$ appears to be the most subtle case of all. Unlike the other three cases where either the isoparametric hypersurface is homogeneous for 
$(m_1,m_2)=(4,5)$, or
one is homogeneous and the other is not for $(m_1,m_2)=(3,4)$ or $(6,9)$, the three known examples in this last case are all inhomogeneous and are 
intertwined with the nonassociativity of the octonion algebra. Meanwhile, with $0\leq k\leq m_1-1=6$, {\em a priori}~\eqref{refinedest} gives $8\geq 13-j\geq 2k+1-j$; this becomes much more entangled than the 
previous cases, as we have $j\leq 4$ to handle. To be able to effectively handle the codimension 2 estimate, we may need to introduce a concept more general 
than Condition A. We have made progress in this direction and shall
report on it in the future.

Lastly, we remark that Immervoll~\cite{Im1} gave a different proof of Theorem~\ref{normalcool} by employing isoparametric triple systems Dorfmeister and Neher developed 
in~\cite{DN0}. It appears that the method has not been applicable to the four exceptional cases.

\end{document}